\documentclass{article}
\usepackage{amsfonts,amsmath,amssymb}
\usepackage{graphicx}
\usepackage{enumitem}
\def\Z{\mathbb{Z}}
\def\Q{\mathbb{Q}}

\let\phi\varphi
\def\Aut{\mathop{\mathrm{Aut}}}
\def\Cayley{\mathop{\mathrm{Cay}}}

\newcommand{\footremember}[2]{%
   \footnote{#2}
    \newcounter{#1}
    \setcounter{#1}{\value{footnote}}%
}
\newcommand{\footrecall}[1]{%
    \footnotemark[\value{#1}]
}

\title{Highly arc-transitive digraphs -- structure and counterexamples
\thanks{Emails: {\tt \{mdevos,mohar\}@sfu.ca, samal@iuuk.mff.cuni.cz}}}

\date{} 

\author{Matt DeVos
\footremember{SFU}{Dept. of Mathematics, Simon Fraser University, Burnaby, Canada}
\and 
  Bojan Mohar\footrecall{SFU}\thanks{Supported in part by the
  Research Grant P1--0297 of ARRS (Slovenia), by an NSERC Discovery Grant (Canada)
  and by the Canada Research Chair program.}~\thanks{On leave from:
  IMFM \& FMF, Department of Mathematics, University of Ljubljana, Ljubljana,
  Slovenia.}
  \and
   Robert \v{S}\'{a}mal
   \thanks{Computer Science Institute of Charles University, Prague, Czech Republic.}  \ 
   \thanks{
     Partially supported by Karel Jane\v{c}ek Science \& Research Endowment (NFKJ) grant 201201. 
     Partially supported by grant LL1201 ERC CZ of the Czech Ministry of Education, Youth and Sports. 
     Partially supported by grant GA \v{C}R P202-12-G061.
   }
}

\newtheorem{theorem}{Theorem}[section]
\newtheorem{lemma}[theorem]{Lemma}

\newtheorem{proposition}[theorem]{Proposition}

\newtheorem{claim}[theorem]{Claim}
\newtheorem{conjecture}[theorem]{Conjecture}
\newtheorem{question}[theorem]{Question}

\newenvironment{proof}{\par\medskip\noindent{\bf Proof\ }}
  {\hskip 2cm\unskip\hbox{}\hfill$\Box$\par\bigskip}
\newenvironment{proofof}{\par\medskip\proofof}
  {\unskip\hfill$\Box$\par\bigskip}
\def\proofof #1{\noindent{\bf Proof of #1:\hskip 0.5em}}

\begin{document}

\maketitle

\begin{abstract}
Two problems of Cameron, Praeger, and Wormald [Infinite highly arc transitive digraphs
and universal covering digraphs, Combinatorica (1993)] are resolved.
First, locally finite highly arc-transitive digraphs with universal reachability relation 
are presented. Second, constructions of two-ended highly arc-transitive digraphs are provided, 
where each `building block' is a finite bipartite digraph that is not a disjoint union 
of complete bipartite digraphs.
Both of these were conjectured impossible in the above-mentioned paper. 
We also describe the structure
of two-ended highly arc-transitive digraphs in more generality, heading towards a characterization
of such digraphs. However, the complete characterization remains elusive.
\end{abstract}

\section{Introduction}

A digraph $D$ consists of a set of vertices $V(D)$ and arcs (also termed edges) $E(D) \subseteq V(D)\times V(D)$. We consider digraphs without loops and rely on standard terminology and
notation as in \cite{Bondy_Murty} or \cite{Diestel_2005}. 
In particular, an edge $(u,v)\in E(D)$ is shortly written as
$uv$ and interpreted as the edge from $u$ to $v$.

An \emph{$s$-arc} in a digraph is an $(s+1)$-tuple of vertices $(v_0, v_1, \dots, v_s)$
such that $v_{i-1} v_i$ is an edge for each $i=1, \dots, s$.
A digraph $D$ is \emph{$s$-arc transitive} if for every two $s$-arcs
$(v_i)_{i=0}^s$, $(v'_i)_{i=0}^s$, there is an automorphism
$f$ of~$D$ such that $f(v_i) = v'_i$ for each $i$.
To exclude trivialities, it is also assumed that $D$ has no isolated vertices
and that every arc of $D$ lies on some $s$-arc.


The notion of $s$-arc transitive digraphs parallels that of $s$-arc transitive undirected
graphs. For those, an $s$-arc corresponds to a nonretracting walk of length~$s$.
Celebrated result of Tutte \cite{tutte_family_1947} states that
a finite 3-regular graph can be $s$-arc transitive only if $s \le 5$.
Weiss \cite{weiss_nonexistence_1981} extended this (using the classification of
finite simple groups) to finite $r$-regular graphs ($r>2$); these can be $s$-arc transitive only if $s \le 8$.
(Somewhat trivially, cycles are $s$-arc transitive for every~$s$.)

A digraph is \emph{highly arc-transitive} if it is $s$-arc transitive
for every $s\ge0$. As one may expect, this is very demanding definition.
Indeed, the only connected \emph{finite} highly arc-transitive digraphs
are the directed cycles (including cycles of length 1 and 2). Among infinite digraphs, the number
of highly arc-transitive ones is much larger. Still, they are
rather restricted, which makes the constructions nontrivial,
and one may hope to characterize all such digraphs, at least to some extent.

The motivation to study highly arc-transitive digraphs does not come solely from combinatorics.
There is an intimate connection to totally disconnected locally compact groups that is
presented in M\"oller \cite{moller_CJM_2002}, see also Malni\v{c} et al.~\cite{malnic_highly_2005}.

An obvious infinite highly arc-transitive digraph is the two-way-infinite
directed path, which we shall denote by $Z$. Another immediate example is obtained
when we replace each vertex of~$Z$ by an independent set of size $k$ and
every arc by a (directed) complete bipartite graph $\vec K_{k,k}$ --
formally this is the \emph{lexicographic product} $Z[\overline{K}_k]$
with $\overline{K}_k$ denoting the graph with $k$ vertices and no edges.
Confirm also Lemma~4.3 and Theorem~4.5 in~\cite{cameron_praeger_wormald_1993} for more
on products and high arc-transitivity.


The question of what other highly arc-transitive
digraphs exist has started a substantial amount of research.
The question was originally considered by Cameron, Praeger, and Wormald \cite{cameron_praeger_wormald_1993}.
They presented some nontrivial constructions (details can be found in Section~\ref{construction})
and worked on ways to describe all highly arc-transitive digraphs.
One approach to this involves the \emph{reachability relation}.

Given a digraph $D$, an \emph{alternating walk} is a sequence $(v_0, v_1, \dots, v_s)$
of vertices such that $v_i v_{i+1}$ and $v_i v_{i-1}$ are arcs of $D$
either for all even $i$ or for all odd $i$; informally, when visiting
the vertices $v_0,v_1, \dots, v_s$, we use the arcs of $D$ alternately in the forward
and backward direction. When $e$, $e'$ are two arcs of $D$, we say that
$e'$ is \emph{reachable} from $e$, in symbols $e \sim e'$, if there is an alternating walk which has
$e$ as the first arc and $e'$ as the last one.
One can easily see that this
is an equivalence relation. Moreover, this relation is preserved by any digraph automorphism.
Thus, whenever $D$ is 1-arc transitive, then the digraphs induced by
the equivalence classes are isomorphic to a fixed digraph, which will be denoted by
$R(D)$ ($R$ stands for reachability).

It is shown in \cite{cameron_praeger_wormald_1993} that if the reachability relation has more than
one class, then $R(D)$ is bipartite and a construction
is presented that, for an arbitrary directed bipartite digraph $R$, gives
a highly arc-transitive digraph $D$ with $R(D) \simeq R$. In fact,
a universal cover for all such digraphs is constructed.
Thus a question arises, whether there are highly arc-transitive digraphs
for which the reachability relation is \emph{universal} (by which it
is meant that there is just one equivalence class), as this
approach to classify highly arc-transitive digraphs would not work for them.
Actually, such digraphs are rather easy to construct if we allow
infinite degree. One example would be the digraph $Q$ whose vertex set are all
rational numbers, $V(Q)=\Q$, and two vertices $u,v$ are adjacent if $u<v$.
So, the following question was asked in~\cite{cameron_praeger_wormald_1993}.

\begin{question}\label{q:universal}
  Is there a locally finite highly arc-transitive digraph with universal
  reachability relation?
\end{question}

In Section~\ref{universal} we present a construction
of such digraphs --- showing, in effect, that highly arc-transitive digraphs
form a richer class of digraphs than one might expect.

Many highly arc-transitive digraphs possess a homomorphism onto $Z$. That is a
mapping $f:V \to \Z$ such that for every edge $uv$ we have $f(v)=f(u)+1$. This is
called \emph{property Z} in~\cite{cameron_praeger_wormald_1993}, and the authors ask, whether
all locally finite highly arc-transitive digraphs have this property.
The first examples of locally finite highly arc-transitive digraphs without property Z were constructed by Malni\v{c} et al.\ in \cite{malnicetal_2002}.
Our digraphs with universal reachability relation provide further examples, as a digraph with property Z has infinitely many reachability classes.

\bigskip

Another approach to classify highly arc-transitive digraphs is to use the
number of \emph{ends}. (See \cite{Diestel_2005} for the definition of an end of a graph.)

It is well known that every infinite vertex-transitive graph, and hence also every
highly arc-transitive digraph, has 1, 2, or infinitely many ends.
An example with two ends is $Z$, with infinitely many ends a tree
(where the in-degree of all vertices is some constant $d^-$ and
the out-degree of all vertices is some constant $d^+$).
An example of a highly arc-transitive digraph with just one end is~$Q$. 
Locally finite examples are known, but they are harder to construct. 
In a few words, one can construct them as horocyclic products of trees, see \cite{moeller_descendants_2002} 
for details. 

Let us focus on two-ended digraphs. This class includes the afore-mentioned basic
examples $Z$ and $Z[\overline{K}_k]$, as well as a more
complicated construction by McKay and Praeger \cite[Remark~3.4]{cameron_praeger_wormald_1993} that is also
discussed in our Section~\ref{construction} as Construction~1. This construction was
generalized in~\cite[Definition 4.6]{cameron_praeger_wormald_1993}. 

Based on their generalization and the lack of other examples, it was conjectured 
in \cite{cameron_praeger_wormald_1993} that for each connected 
highly arc-transitive digraph $D$ with two ends, the reachability
digraph $R(D)$ is either infinite, or
a complete bipartite digraph. We disprove this conjecture in Section~\ref{construction},
where we present several constructions that behave in a more complicated way. 
Independently from us, Christoph Neumann has constructed 
counterexample to Conjecture~\ref{simplehats} using a different method. 

Finally, in Section~\ref{structure} we work towards
characterizing all two-ended highly arc-transitive digraphs. 
We show, in particular, that every such digraph 
either admits a quotient by which we can reduce it to a simpler structure, 
or some lexicographic product $G[\overline{K}_k]$ (digraph $G$ with cloned vertices) 
can be constructed by a rather complicated Construction~4 described in Section~\ref{structure}.
This construction uses a finite digraph with colored edges as a `template'.
While this construction provides many complicated new examples and is shown to be universal
(upto cloning of vertices), we are lacking full understanding of when precisely it gives rise to 
a highly arc-transitive digraph.


\section{Highly arc-transitive digraphs with universal reachability relation}
\label{universal}

The following result answers Question~\ref{q:universal} in the affirmative.

\begin{theorem} \label{main}
There is a locally finite highly arc-transitive digraph for which
the reachability relation is universal.
In fact, for every composite integer $d\ge 4$ there is
such digraph with all in-degrees and all out-degrees equal to~$d$.
\end{theorem}

\begin{figure}[ht]
  \begin{center}
    \includegraphics[scale=1,page=1]{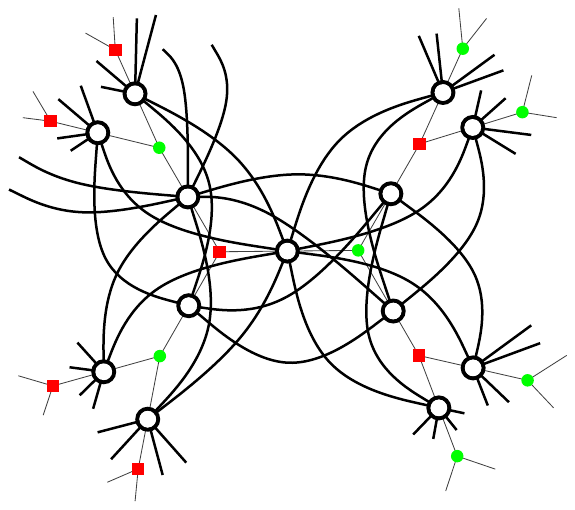}
  \end{center}
  \caption{The digraph $G_{3,3}$ -- a part of the digraph
   (with the underlying tree), without directions of edges.
   Vertices of the set~$A$ are small circles, vertices of~$B$ are squares.} 
  \label{fig:constrtree}
\end{figure}

\begin{figure}
    \includegraphics[scale=1,page=2]{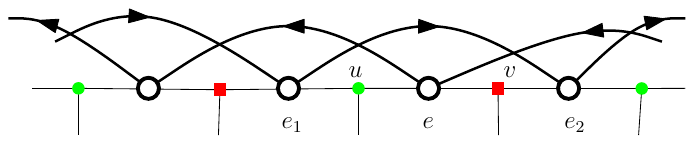}
  \caption{The digraph $G_{3,3}$ -- description of the direction of edges.
   Vertices of the set~$A$ are small circles, vertices of~$B$ are squares.
  }
  \label{fig:constr}
\end{figure}

\def\bfe{{\bf e}}

\begin{proof}
Pick integers $a, b \ge 3$. We will construct a digraph $G_{a,b}$, in which 
every vertex has in- and out-degree equal to $(a-1)(b-1)$ and which satisfies 
the conditions of the theorem. 
Let $T=T_{a,b}$ be the infinite tree with vertex set $A \mathbin{\dot{\cup}} B$, where
every vertex in~$A$ has $a$ neighbours in~$B$, and
every vertex in~$B$ has $b$ neighbours in~$A$.
Next, we define the desired digraph with $V(G_{a,b}) = E(T_{a,b})$.
For each $e=uv \in E(T_{a,b})$, where $u \in A$, $v\in B$, we add an arc from each
$e_1\ne e$ incident with $u$ to each $e_2\ne e$ that is incident with $v$.
For each such pair $e_1,e_2$ we put $c(e_1,e_2) := e$. We let $G=G_{a,b}$ be the resulting
digraph; in Fig.~\ref{fig:constrtree} and~\ref{fig:constr} we display part of~$G_{3,3}$.

First we prove that $G$ is highly arc-transitive. Suppose
$\bfe=(e_0,e_1,\dots,e_s)$ is an $s$-arc in~$G$, and let
$P(\bfe)$ be $e_0, c(e_0,e_1), e_1, \dots, c(e_{s-1},e_s), e_s$,
the corresponding path in~$T$.
Now let $\bfe'$ be another $s$-arc in~$G$.
Obviously $P(\bfe)$ and $P(\bfe')$ are paths in $T$ of the
same length, both starting at a vertex of~$B$. Consequently, there is an
automorphism $\varphi$ of~$T$ that maps $P(\bfe)$ to $P(\bfe')$.
The mapping that $\varphi$ induces on $E(T) = V(G)$ is clearly an automorphism of~$G$
that sends $\bfe$ to $\bfe'$.

We still need to show that the reachability relation of $G$~is universal. Suppose
$e, e' \in V(G)$ are adjacent as edges in~$T$, and that $h$ (resp. $h'$) is an
arc of~$G$ starting at $e$ (resp. $e'$). We will show that $h \sim h'$; this is
clearly sufficient. Assume first that $e$ and $e'$ share a vertex of~$A$.
Let $h_1$, $h_2$ be arcs of~$G$ as depicted in the left part of Fig.~\ref{fig:equiv}
(recall that $a \ge 3$). Obviously $h, h_1, h_2, h'$ is an alternating walk, thus
$h \sim h'$. Secondly, assume $e$ and $e'$ share a vertex of~$B$.
In this case pick arcs $h_1$, $h_2$ according to the right part of Fig.~\ref{fig:equiv},
utilizing that $b \ge 3$. Now $h \sim h_1$ and $h_2 \sim h'$ according to the
first case. This shows that $h_1 \sim h_2$ and completes the proof.
\end{proof}

\begin{figure}[ht]
 \centerline{
   \hss
   \parbox[c]{7cm}{\includegraphics[scale=1,page=4]{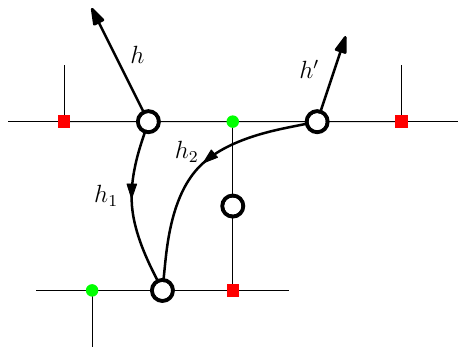}}
   \hfill
   \parbox[c]{6cm}{\includegraphics[scale=0.9,page=3]{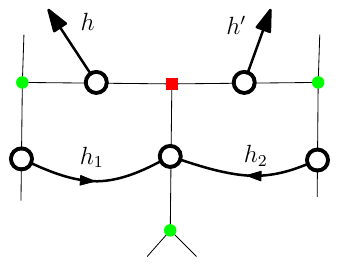}}
   \hss
 }
 \caption{Two arcs of~$G_{a,b}$ that start at adjacent edges of~$T_{a,b}$
   are equivalent.}
  \label{fig:equiv}
\end{figure}

\paragraph{Remark:}
It is known that highly arc-transitive digraphs with universal reachability relation
do not exist if indegrees $d^-$ and outdegrees $d^+$ are not the same \cite{praeger_homomorphic_1991},
and neither they exist if $d^+ = d^-$ is a prime \cite{cameron_praeger_wormald_1993}.
However, whenever $d^+ = d^-$ is not a prime, it can be written as $(a-1)(b-1)$
for $a, b \ge 3$, so Theorem \ref{main} provides an example of such a digraph.

Note that the structure of the digraph $G_{a,b}$ can also be described as
follows. Consider a partition of the vertices of~$K_{a,a(b-1)}$ into $a$ copies of
a star, $K_{1,b-1}$. Let us denote these copies by $S_1$, \dots, $S_a$.
We let $H$ be $K_{a,a(b-1)} - \cup_i E(S_i)$.
Then we take countably many copies of~$H$ and glue them together
(in a tree-like fashion) by identifying in pairs the sets corresponding to some of the~$S_i$'s.
From this description it is immediate that $G_{a,b}$ has universal
reachability relation.


\section{Two-ended constructions}
\label{construction}

As mentioned in the introduction, a highly arc-transitive digraph can have 1,~2, or infinitely
many ends; in the rest of this paper we concentrate on the case of two ends.
It is not hard to show (see the proof of Proposition~\ref{Zsystem}) that every
two-ended 1-arc transitive digraph $D$ has the following structure:
the vertices can be partitioned as $V(D) = \bigcup_{i=-\infty}^\infty V_i$ and all arcs
go from some $V_i$ to $V_{i+1}$. Moreover, if $D$ is also vertex-transitive,
then each of the induced digraphs $B_i=D[V_i \cup V_{i+1}]$
is isomorphic to a fixed bipartite `tile' $B$.
If $B$ is a complete bipartite digraph $\vec K_{k,k}$, we get the basic
example $Z[\overline{K}_k]$. If $B$ is not the complete bipartite digraph,
then $D$ is not determined just by $B$, as we need to specify how are
consecutive copies $B_i$ and $B_{i+1}$ of $B$ `glued' together at $V_{i+1}$.
It is easy to see that all components of $B$ are isomorphic to the reachability digraph $R(D)$.
The following was conjectured in \cite{cameron_praeger_wormald_1993}.

\begin{conjecture}[Cameron, Praeger, and Wormald \cite{cameron_praeger_wormald_1993}]
\label{simplehats}
If $D$ is a con\-nect\-ed highly arc-transitive digraph such that there exists a homomorphism
$f: D \to Z$ and $f^{-1}(0)$ is finite, then $R(D)$ is a complete
bipartite digraph.
\end{conjecture}

Next, we describe several constructions. We start with the one found
by McKay and Praeger \cite[Remark~3.4]{cameron_praeger_wormald_1993}, that, while nontrivial,
concurs with the above conjecture. Next, we shall present our construction (Construction~2), 
disproving the conjecture. Continuing, we shall provide some more complicated examples.
In Section~\ref{structure}, we introduce a very general construction and provide some evidence
that this construction essentially describes all two-ended highly arc symmetric digraphs.

We want to mention here that recently (and independently) 
Christoph Neumann has constructed~\cite{Neumann} counterexample to Conjecture~\ref{simplehats} 
using a different setting. His method (as well as ours) allows for many modifications 
and extensions, however his and ours smallest counterexamples are isomorphic.

\paragraph{Construction 1 (McKay and Praeger \cite[Remark~3.4]{cameron_praeger_wormald_1993})}
  Let $S$ be a finite set, $n$ a positive integer, and let $V = \Z \times S^n$.
  The set $V$ is considered as the vertex-set of the digraph in which two
  vertices $a = (i, a_1, \dots, a_n)$ and $b = (i+1, b_1, \dots, b_n)$ are
  adjacent if $a_j = b_{j+1}$ for each $j=1, \dots, n-1$; no other edges are
  present.

  Here, the digraph $B$ is a disjoint union of complete bipartite digraphs
  (more precisely, $B$ consists of $|S|^{n-1}$ copies of $\vec K_{|S|, |S|}$), thus $R(D)$ is $\vec K_{|S|, |S|}$.
  The fact that this is a highly arc-transitive digraph is easy to show directly, but also follows from our next constructions.

\paragraph{Construction 2}
  Let $T$ be a ``template'' -- an arc-transitive digraph that is bipartite with parts $A_1$, $A_2$,
  all arcs directed from $A_1$ to $A_2$. Let $D$ be the digraph with vertex-set
  $V = \Z \times A_1 \times A_2$, in which two
  vertices $(i,a_1, a_2)$, $(i+1, b_1, b_2)$ are connected if $(a_1,b_2) \in E(T)$.
  We define $V_i = \{i\} \times A_1 \times A_2\subset V$.

  It is clear from the definition that the digraph joining $V_i$ and $V_{i+1}$
  is isomorphic to the bipartite digraph $B$ which
  is obtained from $T$ by taking $|A_2|$ copies of each vertex in $A_1$
  and $|A_1|$ copies of each vertex in $A_2$, and replacing each arc in $T$ by the
  complete bipartite digraph~$\vec K_{|A_2|,|A_1|}$.
  If $T$ is connected, then $B$~is isomorphic to~$R(D)$.
  As shown by Theorem \ref{thm:Constr3hat}, the resulting digraph is highly arc-transitive.
  Thus, by taking $T$ to be $\vec K_{3,3}$ minus a matching (alternately oriented 6-cycle) we get
  a counterexample to Conjecture~\ref{simplehats}.

\paragraph{Construction 3} The next construction is a common generalization
  of Constructions 1 and~2.
  Let $T$ be a $(t-1)$-arc-transitive template digraph, with vertices in $t$ ``levels'',
  $A_1$, \dots, $A_t$ and each arc leading from $A_j$ to $A_{j+1}$ for some $j$.
  We shall denote by~$T_i$ the subgraph of~$T$ induced by~$A_i \cup A_{i+1}$. 
  Suppose that each vertex $v\in V(T)\setminus A_1$ has in-degree at least 1,
  and each vertex $v \in V(T)\setminus A_t$ has out-degree at least~1.
  Now, define a digraph $D=D(T)$ with vertex-set
  $V = \Z \times A_1 \times A_2 \times \cdots \times A_t$, in which two
  vertices $a=(i, a_1, a_2, \dots, a_t)$ and $b=(i+1, b_1, b_2, \dots, b_t)$ are
  adjacent if $(a_j,b_{j+1}) \in E(T)$
  for each $j =1, \dots, t-1$, and no other edges are present in $D$.
  Clearly, for $t=2$ we get Construction~2.
  Construction~1 of McKay and Praeger is a special case of this one, with $T$~consisting
  of~$|S|$ disjoint paths.

\begin{theorem}
\label{thm:Constr3hat}
If $T$ is as in Construction~3, then the digraph $D(T)$ is connected and highly arc-transitive.
If all graphs~$T_i$ are connected then $R(D(T))$ (equivalence class of the reachability relation) 
is isomorphic to the subgraph of~$D(T)$ induced by vertices $\{0,1\} \times A_1 \times \cdots \times A_t$. 
\end{theorem}

\begin{proof}
As before, let $V_i = \{i\} \times A_1 \times \cdots \times A_t$. For a vertex $a\in V_i$, we denote
its $j$-th component by $a_j$, starting with $a_0=i$ and having $a_j\in A_j$ for $j=1,\dots,t$.
First we show that $D=D(T)$ is connected. It is easy to see that the following statement
suffices for this:
for every $a \in V_0$ and $b \in V_t$, there is a directed $(a,b)$-path.
In order to prove this, observe that every vertex of $T$ is a part of at least one directed
path with $t$ vertices. Let $P_i$ ($Q_i$, resp.) be such a path containing $a_i$ ($b_i$, resp.).
We let $P_{i,j}$ denote the $j$-th vertex on $P_i$, so that $P_{i,i} = a_i$
(and, similarly, $Q_{i,i} = b_i$). Now we define vertices $c_0, c_1, \dots, c_t$ in $D$
forming a directed path. For $i=0,\dots,t$ we set $c_{i,0}=i$ and
$$
 c_{i,j} = \begin{cases}
             Q_{j-i+t,j} & \hbox{if $1 \le j \le i$,} \\
             P_{j-i,j} & \hbox{if $i < j \le t$. }
           \end{cases}
$$
Clearly, $c_{0,j}=P_{j,j}$ for $1\le j\le t$, thus $c_0=a$.
Similarly, $c_{t,j}=Q_{j,j}$ for $1\le j\le t$, thus $c_t=b$.
Comparing $c_i$ and $c_{i+1}$ ($0\le i<t$), we see that $c_{i,j}$ and $c_{i+1,j+1}$ are
consecutive vertices of $T$ on the same path, $P_{j-i}$ or $Q_{j-i+t}$. This shows
that $c_i$ and $c_{i+1}$ are adjacent in $D$, and shows that $D$~is connected.

Next, we study the reachability relation. Let $B = T[V_0 \cup V_1]$. Obviously, 
no alternating walk can leave~$B$; we only need to show, that any two edges  
in~$B$ are connected by an alternating walk. Let $xy$ and $uv$ be two such edges. 
Each of the bipartite graphs~$T_i$ is connected (by assumption), thus every two 
of its edges are connected by an alternating walk. We will use this for edges 
$x_i y_{i+1}$ and~$u_i v_{i+1}$ and let $a_i(j) b_{i+1}(j)$ be the $j$-th edge of 
this walk (with $j=0$ corresponding to the starting edge). 
We may assume that all of these walks are of the same length and 
each of them starts by ``fixing the head of the edge'': that is, 
for every~$i=1, \dots, t-1$ we have $b_{i+1}(0) = b_{i+1}(1) = y_{i+1}$. 
Put $b_1(0) = b_1(1) = y_1$, $b_1(j)=v_1$ for $j>1$. Put $a_t(0) = x_t$, $a_t(j)=u_t$ for $j>0$. 
Finally, put $a_0(j) = 0$ and $b_0(j)=1$ for all $j$. 
By construction, edges~$a(j)b(j)$ form an alternating walk in~$D$ connecting~$xy$ and~$uv$. 
It follows that $R(D)$~is isomorphic to~$B$.

To prove that $D$~is highly arc-transitive, we describe some of its automorphisms. 
A trivial one is a shift in the first coordinate, $\tau: a \mapsto (a_0+1, a_1, \dots, a_t)$.
More interesting automorphisms are those that preserve the levels $V_i$.
They come from the automorphisms of $T$.
Let $\phi \in Aut(T)$. Let $\psi:V(D)\to V(D)$ be the mapping that applies $\phi$
on the $j$-th coordinate in $V_j$ for $j=1,\dots,t$ and is identity elsewhere.
We shall show that $\psi$ is an automorphism of $D(T)$. Suppose $ab \in E(D)$,
but $\psi(a)\psi(b)\notin E(D)$. Since $\psi$ preserves the sets $V_i$, from 
the construction of~$D(T)$ (Construction~3) it follows, that
there exists $j = 1, \dots, t-1$ such that $\psi(a)_j \psi(b)_{j+1} \notin E(T)$.
By the definition of $\psi$, we conclude that $a\in V_j$ and $b\in V_{j+1}$. Moreover,
$\psi(a)_j = \phi(a_j)$ and $\psi(b)_{j+1} = \phi(b_{j+1})$.
By assumption, $ab\in E(D)$, so $a_j b_{j+1}$ is an edge of $T$,
and as $\phi$ is an automorphism of $T$, $\phi(a_j)\phi(b_{j+1})$ is an edge of $T$ as well.
This contradicts our assumption and proves that $\psi$ is a homomorphism $D\to D$.
Since $\psi$ is invertible (as $\psi^{-1}$ comes from the inverse automorphism $\phi^{-1}$
of~$T$ by the same construction as~$\psi$ from~$\phi$), we conclude that $\psi$ is an automorphism of~$D$.

Let $\Psi$ be the set of all automorphisms $\psi$ that are obtained from $\phi\in Aut(T)$ in
the way as described above. We claim that the group generated by $\tau$ and $\Psi$
acts transitively on the $s$-arcs in $D$ (for every~$s$). Let $(v_i)_{i=0}^s$, $(v'_i)_{i=0}^s$ be two
$s$-arcs in $D(T)$. By applying $\tau$ or $\tau^{-1}$, we may assume that
$v_0\in V_0$ and $v'_0 \in V_0$, and thus also $v_i,v'_i \in V_i$ for each $i$.
We imagine coordinates of the two arcs written in a grid: all coodinates 
of~$v_i$ ($v'_i$, resp.) in the $i$-th row. We are going to find an automorphism~$\psi_k$ 
of~$D$ such that $\psi_k(v)$ is closer to~$(v')$ than $(v)$. We shall do this by applying 
an automorphism~$\psi \in \Psi$ on an appropriate diagonal (the first diagonal in which 
$(v)$ and $(v')$ differ). Now, we make this idea precise: 

If $v_i = v'_i$ for each $i$, then we are done; otherwise find $i$ and~$j$ so that
\begin{equation*}
   (v_i)_j \ne (v'_i)_j \quad \hbox{and}\quad k=i-j \hbox{ is minimal.} \qquad (*)
\end{equation*}
We put $a_\ell = v_{\ell+k,\ell}$ if $0\le \ell+k\le s$ and $1 \le \ell \le t$. 
After that, we pick $a_\ell\in A_\ell$ (for $\ell$ such that $1 \le \ell \le t$ but 
$\ell+k < 0$ or $\ell+k > s$). The only condition now is that $a_\ell a_{\ell+1}$ 
is an arc for all~$\ell = 1, \dots, t-1$, so that $(a_\ell)_{\ell=1}^t$ is a $(t-1)$-arc in $T$.
Similarly, we define $a'_\ell$ from $v'$.
Now $(a_\ell)$, $(a'_\ell)$ are two $(t-1)$-arcs in $T$, thus (by the symmetry assumptions on~$T$)
there is an automorphism $\phi$ of $T$ such that $\phi(a_\ell) = a'_\ell$ for $\ell=1,\dots,t$.
Let $\psi$ be the automorphism of~$D$ corresponding to~$\phi$,
and let $\psi_k=\tau^k \psi \tau^{-k}$. The mapping~$\psi_k$ permutes the $j$-th coordinate in $V_{k+j}$.
Observe that $s$-arcs $(\psi_k(v_i))_{i=0}^s$ and $(v'_i)_{i=0}^s$ are
closer (so that we get larger value of $k$ in $(*)$) than for $(v_i)$ and $(v'_i)$.
So, after repeating this
procedure at most $s+t$ times we map one $s$-arc to the other.
\end{proof}

As the requirements on the template $T$ are rather strong, let us describe
a nice source of nontrivial templates. Consider a finite affine or projective space,
$AG(n,q)$ or $PG(n,q)$. Let $A_i$ be the family of subspaces of dimension $i-1$. We let
the arcs denote incidence, i.e., $(x,y)$ is an arc if and only if $x$ is a subspace of $y$
of codimension 1. This gives a template with $t=n-1$. A $(t-1)$-arc corresponds
to a flag (that is, a sequence of a subspaces one contained in another, one in each dimension).
It is not hard to show that the geometric space is flag-transitive, which implies the following. 

\begin{claim}
The template just described satisfies the conditions of Construction~3. 
\end{claim}

A natural question remains: does Construction 3 give some highly arc-transitive digraphs
that cannot be obtained by Construction~2? The answer is positive. To prove it, let us first
define the notion of \emph{clones}. Given a digraph, we call vertices $x$, $x'$ \emph{right clones}, if
they have the same outneighbours ($xy$ is an edge if and only if $x'y$ is an edge); we call them
\emph{left clones} if they have the same inneighbours. It is not hard to show that
in a highly arc-transitive digraph, all vertices have the same number $c^+$ of right clones
and the same number $c^-$ of left-clones. In Construction~2 we have $c^+ \ge |A_2|$
and $c^- \ge |A_1|$, so $c^+ c^- \ge |V_0|$.
On the other hand, using Construction~3 with a template $T$ from finite geometries we
have $c^+ = |A_t|$ and $c^- = |A_1|$. In particular, when $t>2$, we have
$c^+ c^- < |V_0|$. This shows that these highly arc-transitive digraphs cannot be
obtained by Construction~2.

\section{Structure in the two-ended case}
\label{structure}

The goal of this section is to prove a structural result concerning two-ended highly arc-transitive digraphs. Our
structure theorem will show that every two-ended highly arc-transitive digraph either admits a quotient by which we can
reduce it to a simpler structure, or up to vertex cloning, can be represented using a generalized construction which we
describe next.

\paragraph{Construction 4}
We define a
\emph{coloured template} to be a digraph $K$ equipped with a possibly improper colouring of the edges $\phi : E(K) \rightarrow \{1,\ldots,t \}$ and also equipped with a distinguished partition of
the vertices into sets $V_0, V_1, \ldots, V_m$ so that every edge goes from a point in $V_i$
to a point in $V_{i+1}$ for some $0 \le i < m$.  Given such a template $K$, we define the digraph $\widehat{K}$
to have vertex set ${\mathbb Z} \times V_0 \times V_1 \times\cdots \times V_m$ and an edge from $(i,x_0, x_1, \ldots, x_m)$ to $(i+1,y_0,y_1,\ldots,y_m)$ whenever all of the
arcs $(x_0, y_1), (x_1, y_2), \ldots ,(x_{m-1}, y_m)$ are present in $K$ and all have
the same colour.

\bigskip


It is easy to see that Construction 4 generalizes Construction 3. However, the digraphs $\widehat{K}$ are not always highly arc-transitive. In this section we shall prove
that all two-ended highly arc-transitive digraphs can be described by using Construction 4 combined with vertex-cloning operation. 
The proof of this will be built up slowly in a series of small lemmas.

Throughout this section, we shall always assume that $G$ is a highly arc-transitive digraph%
\footnote{
  Let us note that some for our lemmas hold more generally. 
}
so that the underlying undirected graph is connected and has two ends.  For any partition
${\mathcal P}$ of the vertices, we let $G^{\mathcal P}$ denote the digraph obtained from $G$
by identifying the vertices in each block of ${\mathcal P}$ to a single new vertex and then deleting
any parallel edges.  We say that a system of imprimitivity ${\mathcal B}$ is a
${\mathbb Z}$-\emph{system} if $G^{\mathcal B}$ is isomorphic to two-way-infinite directed
path.  In this case the blocks of ${\mathcal B}$ can be enumerated
$\{B_i\}_{i \in {\mathbb Z} }$ so that every edge has its tail in $B_i$ and its head in $B_{i+1}$ for some
$i \in {\mathbb Z}$.  Note that in this case, we have that for every $\phi \in Aut(G)$ there exists
$j \in {\mathbb Z}$ so that $\phi(B_i) = B_{i+j}$ for every $i \in {\mathbb Z}$.

Some of the results that follow, or parts of their proofs, can be found in
\cite{cameron_praeger_wormald_1993} or in \cite{moeller_descendants_2002}.
We include them for completeness.

\begin{proposition}\label{Zsystem}
Every connected two-ended 2-arc transitive
digraph has a unique ${\mathbb Z}$-system ${\mathcal B}$.
Furthermore, $\mathcal B$ has finite blocks of imprimitivity, and every system of
imprimitivity with finite blocks is a refinement of ${\mathcal B}$.
\end{proposition}


\begin{proof}
Every connected vertex-transitive two-ended digraph has a system of imprimitivity ${\mathcal B}$ with finite blocks and an
(infinite) cyclic relation on ${\mathcal B}$ which is preserved by the automorphism group; this follows, for instance
from Dunwoody's theorem \cite{dunwoody_cutting_1982} on cutting up graphs.
Enumerate the blocks $\{ B_i \}_{i \in {\mathbb Z}}$ so that this
cyclic relation associates $B_i$ with $B_{i-1}$ and $B_{i+1}$ for every $i \in {\mathbb Z}$.  Now, it follows from the
assumption that the digraph $G$ is arc-transitive that there exists a fixed integer $k$ so that every edge with one end in $B_i$ and
one end in $B_j$ satisfies $|i-j| = k$.  It then follows from the connectivity of the underlying graph that $k=1$.  So,
every edge has its ends in two consecutive blocks of $\{ B_i \}_{i \in {\mathbb Z}}$ .

Note that every vertex $x \in B_i$ must be adjacent in the underlying undirected graph
to both a vertex in $B_{i-1}$ and in $B_{i+1}$ (otherwise every vertex would behave
similarly, and the graph would be disconnected).  Suppose (for a contradiction) that
there exists a directed path $P$ of length two with vertex sequence $x_0, x_1, x_2$ so
that both $x_0$ and $x_2$ are contained in the same block $B_i$.  Choose a vertex $y$
which is adjacent to $x_1$ in the underlying undirected graph but is not in $B_i$.  Now
either $x_0, x_1, y$ or $y, x_1, x_2$ is the vertex sequence of a directed path of length
two; we let $P'$~denote this path.  It follows immediately that no automorphism can map~$P$ to~$P'$, 
and this contradicts the assumption of 2-arc transitivity.  Therefore, by possibly
reversing our ordering, we may assume that every edge has its tail in some block $B_i$
and its head in $B_{i+1}$.  Thus $\mathcal{B}$ is a $\mathbb{Z}$-system.

For the last part of the theorem, we let ${\mathcal C}$ be a system of imprimitivity with finite blocks, and suppose (for a contradiction) that ${\mathcal C}$ is not a refinement of ${\mathcal B}$.  Choose
a block $C$ of ${\mathcal C}$ and let $i \in {\mathbb Z}$ be the smallest integer with $B_i \cap C \neq \emptyset$ and let $j \in {\mathbb Z}$ be the largest integer with $B_j \cap C \neq \emptyset$
(and note that $i < j$).  Now choose a vertex $u \in B_i \cap C$ and $v \in B_j \cap C$ and choose an automorphism $\phi$ so that $\phi(u) = v$.  It now follows that $\phi(C) = C$ and that $\phi(B_k) = B_{k + j - i}$ for every $k \in {\mathbb Z}$, but this implies that $C$ is infinite, and thus we obtain a contradiction.
Thus, ${\mathcal C}$ must be a refinement of ${\mathcal B}$.  It follows immediately from this that the ${\mathbb Z}$-system $\mathcal B$ is unique.
\end{proof}

In the sequel, we shall work extensively with group actions; our groups shall act on the left.  For clarity, we shall always use upper case Greek letters for groups and lower case Greek letters for elements of groups.  If $\Psi$ is a group and $\Lambda \le \Psi$ we let
$\Psi / \Lambda$ denote the set of left $\Lambda$-cosets in~$\Psi$.
Further, we let $G$ be a connected two-ended highly arc-transitive digraph and
we let ${\mathcal B} = \{ B_i \}_{i \in {\mathbb Z}}$ be its ${\mathbb Z}$-system.

\begin{lemma}
There exists a nontrivial automorphism of\/ $G$ with only finitely many non-fixed points.
\end{lemma}

\begin{proof} Let ${\mathcal B} = \{ B_i \}_{i \in {\mathbb Z}}$ be the
${\mathbb Z}$-system, and suppose that every vertex has outdegree $d$ and that each block of
${\mathcal B}$ has size $k$.  Next, choose an integer $n$ large enough so that
$d^n > (k!)^2$ and consider a directed path $P$ of length $n$ with vertex sequence
$x_0, x_1, \ldots, x_n$ with $x_i \in B_i$.  Now, there are $d^n$ directed paths of length
$n$ which start at the vertex $x_0$, and for each of them, we may choose an automorphism
which maps $P$ to this path.  Since $d^n > (k!)^2$ it follows that there must be two such automorphisms, say $\phi_1$ and $\phi_2$ which give exactly the same permutation of both $B_0$ and $B_n$.  It follows that the automorphism $\psi = \phi_1 \phi_2^{-1}$ is nontrivial, but gives the identity permutation on both $B_0$ and $B_n$.  Now, we define a mapping
$\psi' : V(G) \rightarrow V(G)$ by the following rule
\[ \psi'(x) = \left\{ \begin{array}{cl} \psi(x)	&	\mbox{if $x \in B_1 \cup B_2 \cup \cdots \cup B_{n-1}$} \\
						x	&	\mbox{otherwise.}
						\end{array} \right. \]
It is immediate that $\psi'$ is a nontrivial automorphism which has only finitely many non-fixed points, as desired.
\end{proof}

Based on the above lemma, there exists a smallest integer $\ell$ so that $G$ has a nontrivial automorphism which fixes all but $\ell+1$ blocks from the ${\mathbb Z}$-system pointwise.  It is immediate that every such automorphism must give a non-identity permutation on $\ell+1$ consecutive blocks and the identity on all others.  For every integer $i$, let $\Gamma_i$ denote the subgroup of automorphisms which pointwise fix all blocks of the ${\mathbb
Z}$-system with the (possible) exception of $B_{i-\ell}, B_{i-\ell+1}, \ldots, B_i$.  We let $\Gamma$ denote the subgroup of $\Aut(G)$ generated by $\cup_{i \in {\mathbb Z}} \Gamma_i$.

\begin{lemma}
\label{gamma_prop}
The following statements hold:
\begin{enumerate}[label=(\roman{*})]
\item If  $\alpha \in \Gamma_i$ and $\beta \in \Gamma_j$ with $i \neq j$, then $\alpha$ and $\beta$ commute.
\item If $\phi \in \Aut(G)$ satisfies $\phi(B_0) = B_k$ then $\phi \Gamma_j \phi^{-1} = \Gamma_{j+k}$ for every $j\in {\mathbb Z}$.
\item $\Gamma \triangleleft \Aut(G)$.
\end{enumerate}
\end{lemma}

\begin{proof}
To prove claim~(i), we consider the mapping $\gamma = \alpha \beta \alpha^{-1} \beta^{-1}$.  Since $\alpha$ pointwise fixes all
blocks but $B_{i - \ell }, B_{i- \ell + 1}, \ldots, B_{i}$ and $\beta$ pointwise fixes all blocks but
$B_{j - \ell }, B_{j-\ell + 1}, \ldots, B_{j}$
the map $\gamma$ fixes pointwise any block, which is not in both of these lists.  However,
then $\gamma$ must pointwise fix all but fewer than $\ell+1$ blocks, so $\gamma$ is the identity.

For the second claim, we first note that $\phi(B_i) = B_{i+k}$ for every $i \in {\mathbb Z}$.  Now, for every
$\alpha \in \Gamma_j$ we see that $\phi \alpha \phi^{-1}$ pointwise fixes all blocks
except possibly $B_{j+k-\ell}, B_{j+k - \ell +1}, \ldots, B_{j+k}$ and it follows that
$\phi \alpha \phi^{-1} \in \Gamma_{j+k}$ which proves the claim.

To prove claim~(iii), let $\alpha \in \Gamma$ and express this element as $\alpha = \alpha_1 \alpha_2 \ldots \alpha_m$ where each
$\alpha_i$ is in a subgroup of the form $\Gamma_j$.  Now we have
$$
  \phi \alpha \phi^{-1} = (\phi \alpha_1 \phi^{-1}) (\phi \alpha_2 \phi^{-1}) \ldots (\phi \alpha_m \phi^{-1})
$$
so $\phi \alpha \phi^{-1}$ is also contained in $\Gamma$.
\end{proof}

We call a two-way-infinite directed path a \emph{line}.  The following lemma may be proved with
a straightforward compactness argument, and appears in M\"oller~\cite{moeller_descendants_2002}.

\begin{lemma}
\label{linetrans}
Let\/ $\mathbf{x}, \mathbf{y}$ be lines in $G$ with $x$ a vertex in $\mathbf{x}$ and
$y$ a vertex in $\mathbf{y}$.  Then there exists an automorphism $\phi$ of $G$ which maps
$\mathbf{x}$ to $\mathbf{y}$ and maps $x$ to $y$.
\end{lemma}

\begin{lemma}
\label{decomp}
Let $\Lambda \triangleleft \Aut(G)$ and let\/ ${\mathcal C}$ be the partition of $V(G)$ given by the orbits under the action of $\Lambda$.
\begin{enumerate}[label=(\roman{*})]
\item ${\mathcal C}$ is a system of imprimitivity.
\item If $C,C' \in {\mathcal C}$ and there is an edge from $C$ to $C'$, then every vertex in
$C$ has an outneighbour in $C'$ and every vertex in $C'$ has an inneighbour in $C$.
\item $G^{\mathcal C}$ is highly arc-transitive.
\item If\/ $\mathbf{x}$ is a line in $G$, then the digraph $G_{\mathbf{x}}$ induced by the union of those blocks of ${\mathcal C}$ which contain a vertex in $\mathbf{x}$ is highly arc-transitive.
\item If\/ $\mathbf{x}$ and\/ $\mathbf{y}$ are lines in $G$, then the digraphs $G_{\mathbf{x}}$ and\/ $G_{\mathbf{y}}$ are isomorphic.
\end{enumerate}
\end{lemma}

\begin{proof}
Part~(i) is a standard fact about group actions.  For the proof, let $u,v \in V(G)$ be in the same orbit of $\Lambda$, say $u = \alpha(v)$ for $\alpha \in \Lambda$, and let $\phi$ be any automorphism.  Now, $\phi(u) = \phi \alpha (v) = \phi \alpha \phi^{-1} \phi(v)$. Since $\phi \alpha \phi^{-1}\in \Lambda$, $\phi(u)$ and $\phi(v)$ are also in
the same orbit of~$\Lambda$.

For part~(ii), choose an edge $(u,u') \in E(G)$ with $u \in C$ and $u' \in C'$.  Now, for every $v \in C$ there is an element in $\Lambda$ that maps $u$ to $v$.  Since this element must fix $C'$ setwise,
it follows that $v$ has an outneighbour in $C'$.  A similar argument
shows that every point in $C'$ has an inneighbour in $C$.

To prove~(iii), we let  $C_1, C_2, \ldots, C_k$ and $C_1', C_2', \ldots, C'_k$ be two sequences of blocks of ${\mathcal
C}$ so that both form the vertex set of a directed path in the digraph $G^{\mathcal C}$.  Using part 2 we may choose
vertex sequences $x_1,\ldots,x_k$ and $x'_1, \ldots, x'_k$ in $G$ so that $x_i \in C_i$ and $x'_i \in C'_i$ for 
$1 \le i \le k$ and so that $(x_i, x_{i+1}), (x'_i, x'_{i+1}) \in E(G)$ for $1 \le i \le k-1$.  
It follows from the high arc transitivity
of $G$ that there is an automorphism $\phi$ of $G$ so that $\phi(x_i) = x'_i$ for $1 \le i \le k$.
Then $\phi(C_i) = C'_i$ for $1 \le i \le k$  so $\phi$ induces an automorphism of $G^{\mathcal C}$
that maps $C_1, \ldots, C_k$ to $C'_1, \ldots, C'_k$.  It follows that $G^{\mathcal C}$ is highly
arc-transitive.

For the proof of~(iv), set $X$ to be the union of those blocks of ${\mathcal C}$ which contain a point of $\mathbf{x}$, and set $G'$ to be the digraph induced by $X$.  Now we let $y_1y_2\ldots y_k$ and $y'_1 y_2' \ldots y'_k$ be two paths of length $k-1$ in $G'$. It follows from part 2 that we may extend $y_1 \ldots y_k$ and $y'_1 \ldots y'_k$, respectively, to lines $\mathbf{y}$ and $\mathbf{y'}$ in $G'$.  It now follows from the previous lemma that there is an automorphism $\phi$ of $G$ which maps $\mathbf{y}$ to $\mathbf{y'}$ and further has $\phi(y_i) = y'_i$ for $1 \le i \le k$.  It then follows
that $\phi(X) = X$ so $\phi$ yields an automorphism of $G'$ which sends
$y_1, \ldots, y_k$ to $y'_1, \ldots, y'_k$.  We conclude that $G'$ is highly arc-transitive.

Part~(v) follows easily from Lemma~\ref{linetrans}.
\end{proof}

We define $G$ to be \emph{essentially primitive} if there does not exist
$\Lambda \triangleleft \Aut(G)$ so that the orbits of $\Lambda$ on $V(G)$ generate a proper nontrivial system of
imprimitivity with finite blocks which is not equal to the ${\mathbb Z}$-system.
Parts~3--5 from the  previous lemma show that any two-ended highly arc-transitive digraph which is not essentially
primitive has a type of decomposition into a highly arc-transitive subgraph and a highly arc-transitive quotient.
Although this decomposition does not seem to give us a construction, we will focus in the remainder of this section on
understanding the structure of the essentially primitive digraphs.
Note, however, that we do not know whether this is truly needed. The only examples of highly arc-transitive digraphs
that are not essentially primitive that we are aware of are a disjoint union of two highly arc-transitive digraphs
(rather trivial example) and digraphs obtained by a \emph{horocyclic product} (see~\cite{Woess-horocyc}): we have 
vertices $(i,x,y)$ for each pair of vertices $(i,x)$, $(i,y)$ of the two factors, and vertex 
$(i,x,y)$ is connected to $(i+1,x',y')$ iff both $(i,x)(i+1,x')$ and $(i,y)(i+1,y')$ are arcs in the factors. 
However, such product of
two highly arc-transitive digraphs obtained by our template construction can also be obtained by our construction using a more complicated template.

Continuing with our attempt for a structural characterization we describe orbits of the group~$\Gamma$ 
(see the definition before Lemma~\ref{gamma_prop}). 


\begin{lemma}
\label{prim}
If\/ $G$ is essentially primitive, then the orbits under the action of\/~$\Gamma$ are the blocks
$\{ B_i : i \in {\mathbb Z} \}$ of the $\mathbb Z$-system of\/ $G$.
\end{lemma}

\begin{proof} This follows immediately from Lemma \ref{gamma_prop}.
\end{proof}

Next we shall introduce another useful subgroup of $\Aut(G)$.
Let $\Gamma_k$ ($k\in {\mathbb Z}$) and $\Gamma$ be the subgroups of $\Aut(G)$ introduced
before Lemma \ref{gamma_prop}.
As before, let $\tau$ be an automorphism of $G$ so that $\tau(B_0) = B_1$ (so, more generally,
$\tau(B_i) =  B_{i+1}$), and let $\Phi$ be the subgroup of $\Aut(G)$ which is generated by $\tau$
and $\Gamma$. We will use $\Phi$ to describe our digraph, so let us record some key features of it.
The listed properties follow easily from Lemma \ref{gamma_prop}, and the details of the proof
are left to the reader.

\begin{lemma}
\label{phi_lem}
$\mbox{}$
\begin{enumerate}[label=(\roman{*})]
\item $\tau^{-1} \Gamma_k \tau = \Gamma_{k-1}$ for every $k\in {\mathbb Z}$.
\item $\Gamma \triangleleft \Phi$.
\item $\langle \tau \rangle \cong {\mathbb Z}$.
\item $\Gamma \cap \langle \tau \rangle = \{ 1\}$.
\item $\Phi$ is a semidirect product of $\langle \tau \rangle$ and\/ $\Gamma$.
\end{enumerate}
\end{lemma}


Next we introduce another family of subgroups of $\Phi$.  For every $j \le k$ we define
$\overline{\Gamma}_{j..k}$ to be the subgroup of $\Gamma$ generated by
$\left( \bigcup_{i < j} \Gamma_i \right) \cup  \left( \bigcup_{i>k} \Gamma_{i} \right)$.  Note that
$\overline{\Gamma}_{0..\ell}$ is precisely the subgroup of $\Gamma$ consisting of those automorphisms which act trivially on $B_0$.

\goodbreak 
\begin{lemma}
\label{gammabar}
$\mbox{}$
\begin{enumerate}[label=(\roman{*})]
\item Every coset of $\overline{\Gamma}_{j..k}$ in $\Phi$ has a unique representation as $\tau^m \left( \prod_{i=j}^k \alpha_i \right) \overline{\Gamma}_{j..k}$ where $\alpha_i \in \Gamma_i$ for every $j \le i \le k$
(henceforth we call this the \emph{standard form}).
\item $\tau^{-1} \overline{\Gamma}_{j..k} \tau = \overline{\Gamma}_{j-1..k-1}$
\item If $A \subseteq \tau \Gamma$ then  $\overline{\Gamma}_{j..k} A = A \overline{\Gamma}_{j-1..k-1}$.
\item A set $A \subseteq \tau \Gamma$ satisfies $\overline{\Gamma}_{j..k} A \overline{\Gamma}_{j..k} = A$
if and only if $A \overline{\Gamma}_{j..k-1} = A$.
\end{enumerate}
\end{lemma}

\begin{proof}
The first and the second property follow immediately from the previous lemma.  For the third, choose
$A' \subseteq \Gamma$ so that $A = \tau A'$ and observe that
\[ \overline{\Gamma}_{j..k} A = \overline{\Gamma}_{j..k}\, \tau A' = \tau \, \overline{\Gamma}_{j-1..k-1} A' =
\tau A' \overline{\Gamma}_{j-1..k-1} = A \overline{\Gamma}_{j-1..k-1}. \]
To prove the last property it is enough to observe that for $A \subseteq \tau \Gamma$
$$
  \overline{\Gamma}_{j..k} A \overline{\Gamma}_{j..k} =
  A \overline{\Gamma}_{j-1..k-1} \overline{\Gamma}_{j..k} =
  A \overline{\Gamma}_{j..k-1} \,.
$$
\end{proof}

The only additional ingredients required for our structure theorem are some standard properties of vertex-transitive digraphs.  
Let $\Psi$ be a group, $\Lambda$ a subgroup of $\Psi$, and suppose set $A \subseteq \Psi$ satisfies 
$\Lambda A \Lambda = A$.  Then we define the \emph{Cayley coset digraph}
$\Cayley( \Psi / \Lambda, A)$ to be the digraph whose vertex-set are the left
cosets $\Psi / \Lambda$, where there is an edge from $g\Lambda$ to $h\Lambda$ if and
only if $\Lambda g^{-1}h\Lambda \subseteq A$.  The group
$\Psi$  has a natural action on the vertices by left multiplication, and this action
preserves the edges, and is transitive.  The following well-known result of Sabidussi
\cite{Sabidussi_1964} shows that
every vertex-transitive digraph is isomorphic to a Cayley coset digraph.  Here, if $\Psi$
acts on a set $X$ and $u \in X$ we let $\Psi_u = \{ \gamma \in \Psi : \gamma(u) = u \}$
denote the point stabilizer of $u$.

\begin{proposition}
\label{cay-coset}
Let $H$ be a digraph, let $u \in V(H)$  and let $\Phi \le \Aut(H)$ act transitively on $V(H)$.  Then
there exists $A \subseteq \Phi$ so that $H \cong \Cayley(\Phi/\Phi_u, A)$, and this isomorphism may be chosen so that the vertex $u$ corresponds to the trivial coset $\Phi_u$.
\end{proposition}


Let us recall that \emph{cloning} a vertex in a digraph $G$ means the operation of adding a new
vertex $v'$ whose inneighbours (and outneighbours) are precisely the inneighbours (and the outneighbours) of $v$. Also, let us note that the digraph obtained from $G$ by cloning each vertex
$k-1$ times is just the lexicographic product $G[\overline{K}_k]$ of $G$ with the empty graph on
$k$ vertices.

\begin{proposition}
\label{cay-subgroup}
Let $G = Cayley(\Phi/ \Lambda, A)$ and let $\Lambda' \le \Lambda$ with $[\Lambda : \Lambda'] = k$.  Then $G' = Cayley( \Phi / \Lambda', A)$ is a Cayley coset digraph which is isomorphic to the digraph obtained from $G$ by cloning each vertex $k-1$ times.
\end{proposition}

\begin{proof} (sketch) By definition, in the digraph $G'$ there will be an edge from $Q \in \Phi/\Lambda'$ to
$R \in \Phi / \Lambda'$ if $Q^{-1}R \subseteq A$.  If $R$ and $R'$ lie in the same $\Lambda$-coset then
$Q^{-1}R \Lambda = Q^{-1} R' \Lambda$.  Since $A \Lambda = A$, it follows that there is an edge from
$Q$ to $R$ if and only if there is an edge from $Q$ to $R'$.  So, two vertices which lie in the same $\Lambda$-coset
will have the same inneighbours.  A similar argument shows that they have the same outneighbours.  Thus, $G'$ is
isomorphic to the digraph obtained from $G$ by cloning each vertex exactly $k-1$ times.
\end{proof}

\begin{theorem}
If a two-ended highly arc-transitive digraph $G$ is essentially primitive, then there exists a digraph $G^+$ obtained from $G$
by cloning each vertex the same (finite) number of times and a coloured template $K$ so that
$G^+ \cong \widehat{K}$.
\end{theorem}

\begin{proof}
It follows immediately from Lemma \ref{prim} that the group $\Phi$ generated by $\tau$ and
$\Gamma$ acts transitively on $V(G)$. As before, let $B_i$ ($i\in {\mathbb Z}$) be the blocks
of the $\mathbb Z$-system on $G$. Choose a vertex $u \in B_0$ and apply Proposition \ref{cay-coset} to obtain $A \subseteq \Phi$ so that $G \cong \Cayley(\Phi/\Phi_u, A)$.  Since $\Phi_u$ is the stabilizer of $u$ and $\overline{\Gamma}_{0..\ell}$ is the subgroup of $\Phi$ which fixes every point in $B_0$ we
have $\overline{\Gamma}_{0..\ell} \le \Phi_u \le \Phi$ (and note that this also implies that
$[\Phi_u : \overline{\Gamma}_{0..\ell}]$ is finite).  It now follows from Proposition \ref{cay-subgroup}
that $G^+ =  \Cayley(\Phi/\overline{\Gamma}_{0..\ell}, A)$ is obtained from $G$ by cloning each
vertex the same number of times, so it shall suffice to prove that $G^+$ can be obtained from our
construction.

By assumption, $A$ must satisfy $\overline{\Gamma}_{0..\ell} A \overline{\Gamma}_{0..\ell} = A$ and then it follows from Lemma~\ref{gammabar} that $A \overline{\Gamma}_{0..\ell-1} = A$, so we may partition $A$ into cosets of $\overline{\Gamma}_{0..\ell-1}$ as $\{A_1, A_2, \ldots, A_t\}$. Now,
each $A_q$ also satisfies $\overline{\Gamma}_{0..\ell} A_q \overline{\Gamma}_{0..\ell} = A_q$, so we may define a Cayley coset digraph $G^+_q = \Cayley(\Phi/ \overline{\Gamma}_{0..\ell}, A_q )$ and now $G^+$ is the edge-disjoint union of the digraphs $G^+_1, \ldots, G^+_t$.
We may now view each $q = 1, \dots, t$ as a colour and view $G^+$ as having its edges coloured accordingly.

Fix $1 \le q \le t$ and consider the digraph $G^+_q$ and let $A_q = \tau \left( \prod_{i=0}^{\ell-1} \gamma_i \right) \overline{\Gamma}_{0..\ell-1}$ be represented in standard form.
Let $v = \tau^k \left( \prod_{i=0}^{\ell} \alpha_i \right) \overline{\Gamma}_{0..\ell}$ be a vertex of
$G^+_q$ in standard form.  Within the digraph $G^+_q$, the vertex $v$ will have outneighbours consisting of exactly those cosets of
$\overline{\Gamma}_{0..\ell}$ that are contained in the set
\begin{align*}
v A_q
&= \tau^k \left( \prod_{i=0}^{\ell} \alpha_i \right) \overline{\Gamma}_{0..\ell} \; \tau \left( \prod_{i=0}^{\ell-1} \gamma_i \right) \overline{\Gamma}_{0..\ell-1}  \\
&= \tau^{k+1} \left( \prod_{i=0}^{\ell} \tau^{-1} \alpha_i \tau \right) \tau^{-1} \, \overline{\Gamma}_{0..\ell} \; \tau
\left( \prod_{i=0}^{\ell-1} \gamma_i \right) \overline{\Gamma}_{0..\ell-1} \\
&= \tau^{k+1}  \left( \prod_{i=1}^{\ell} \tau^{-1} \alpha_i \tau \right)
\left( \prod_{i=0}^{\ell-1} \gamma_i \right) \overline{\Gamma}_{0..\ell-1} \\
&=  \tau^{k+1} \left( \prod_{i = 1}^{\ell} \tau^{-1} \alpha_{i} \tau \, \gamma_{i-1} \right)  \overline{\Gamma}_{0..\ell-1}.
\end{align*}
In other words, a vertex $w$ is an outneighbour of $v$ if and only if in standard form
$w = \tau^{k+1} \left( \prod_{i=0}^{\ell} \beta_i \right)  \overline{\Gamma}_{0..\ell}$ where
$\beta_{i-1} = \tau^{-1} \alpha_{i} \tau \gamma_{i-1}$ for every $1 \le i \le \ell$ (and there is no
restriction on $\beta_{\ell}$).
Next we shall define a template $K_q$ with ordered vertex partition
$(\Gamma_{\ell}, \Gamma_{\ell-1}, \ldots, \Gamma_0)$
and an edge from $\delta \in \Gamma_{i}$ to $\epsilon \in \Gamma_{i-1}$ if and only if
$\epsilon = \tau^{-1} \delta \tau \gamma_{i-1}$.  It now follows that
$(v,w)$ is an edge of $G^+_q$ if and only if (using standard form)
$v = \tau^i \alpha_0 \alpha_1 \ldots \alpha_{\ell} \overline{\Gamma}_{0..\ell}$
and $w = \tau^j \beta_0 \beta_1 \ldots \beta_{\ell} \overline{\Gamma}_{0..\ell}$ satisfy $j = i+1$
and $(\alpha_{i}, \beta_{i-1})$ is an edge of $K_q$ for every $1 \le i \le \ell$.  It follows from this
that $G^+_q \cong \widehat{K_q}$ by way of the isomorphism which maps a vertex
$v = \tau^i \alpha_0 \alpha_1 \ldots \alpha_{\ell} \overline{\Gamma}_{0..\ell}$ of $G^+_q$ to the
vertex $(i, \alpha_{\ell}, \alpha_{\ell-1}, \ldots, \alpha_0)$ of $\widehat{K_q}$.

We now define $K$ to be a coloured template with vertex set $\Gamma_1 \cup \Gamma_2 \cup \cdots \cup \Gamma_{\ell}$, vertex partition
$\{ \Gamma_1, \Gamma_2, \ldots, \Gamma_{\ell} \}$, and an edge from $\delta \in \Gamma_i$ to $\epsilon \in \Gamma_{i+1}$ of
colour $q$ if and only if this edge exists in the template $K_q$.  It now follows that
$G^+ \cong \widehat{K}$ which completes the proof.
\end{proof}

\section*{Acknowledgement}

Our attention was drawn to this problem by Norbert Seifter's talk at
the BIRS workshop on Infinite Graphs in October 2007 (Banff, Alberta).

We thank to the anonymous referees for helpful comments leading to improved 
presentation and sharper version of Proposition~\ref{Zsystem}. 

\bibliographystyle{amsplain}
\bibliography{hat}

\end{document}